\documentclass[11pt]{article}
\usepackage{amsfonts}
\usepackage{lipsum}
\usepackage{latexsym}
\usepackage{amsmath}
\usepackage{amssymb}
\usepackage{amsthm}
\usepackage{graphicx}
\usepackage{dsfont}
\usepackage{fullpage}
\usepackage{enumerate}
\usepackage{color}
\usepackage{bbm}
\newtheorem{theorem}{Theorem}[section]
\newtheorem{lemma}[theorem]{Lemma}

\newtheorem{proposition}[theorem]{Proposition}
\newtheorem{corollary}[theorem]{Corollary}
\newtheorem{remark}[theorem]{Remark}

\theoremstyle{definition}

\newtheorem{thmy}{Theorem}
\newenvironment{oldtheorem}{\stepcounter{thm}\begin{thmy}}{\end{thmy}}
\newtheorem*{note*}{Note}

\newcommand{\R}{\mathbb R}
\newcommand{\Rn}{{\mathbb R}^n}

\newcommand{\Sn}{\mathbb{ S}^{n-1}}

\newcommand{\Ha}{{{\cal H}^{n-1}}}

\makeatother

\newcommand\blfootnote[1]{%
  \begingroup
  \renewcommand\thefootnote{}\footnote{#1}%
  \addtocounter{footnote}{-1}%
  \endgroup
}

\begin{document}

\title{\bf A note on the $L_p$-Brunn-Minkowski inequality for intrinsic volumes and the $L_p$-Christoffel-Minkowski problem}
\date{\today}
\medskip
\author{Konstantinos Patsalos, Christos Saroglou}
\maketitle
\begin{abstract}
 The first goal of this paper is to improve some of the results in \cite{BCPR}. Namely, we establish the $L_p$-Brunn-Minkwoski inequality for intrinsic volumes for origin-symmetric convex bodies that are close to the ball in the $C^2$ sense for a certain range of $p<1$ (including negative values) and we prove that this inequality does not hold true in the entire class of origin-symmetric convex bodies for any $p<1$. The second goal is to establish a uniqueness result for the (closely related) $L_p$-Christoffel-Minkowski problem. More specifically, we show uniqueness in the symmetric case when $p\in[0,1)$ and the data function $g$ in the right hand side is sufficiently close to the constant 1. One of the main ingredients of the proof is the existence of upper and lower bounds for the (convex) solution, that depend only $\|\log g\|_{L^\infty}$, a fact that might be of independent interest.   
\end{abstract}

\section{Introduction}
Let $n\geq 2$ be an arbitrary (but fixed) integer and let $\langle\cdot,\cdot\rangle$ be the standard inner product in $\R^n$. For $x\in\mathbb{R}^n$, denote by $|x|:=\sqrt{\langle x,x\rangle}$ the Euclidean norm of $x$ and set $B_2^n:\{x:|x|\leq 1\}$ and $\Sn:=\partial B_2^n$ for the standard unit ball and the unit sphere respectively. The notation $A|B$ stands for the orthogonal projection of a set or a vector $A$ onto the subspace $B$. We denote by $\Ha$ the $n-1$-dimensional Hausdorff measure. For a convex body $K$ in $\R^n$ (that is, a compact convex set with non-empty interior), its {\it support function} $h_K:\R^n\to\R$ is given by
$$h_K(x)=\max_{y\in K}\langle x,y\rangle,\qquad x\in\R^n.$$
It is well known that $h_K$ is a convex and positively homogeneous function. In addition, any convex and positively homogeneous function is the support function of a unique convex body. Furthermore, $h_K\leq h_L$ if and only if $K\subseteq L$ and $K$ contains the origin (resp. in its interior) if and only if $h_K\geq 0$ (resp. $h_K|_{\R^n\setminus\{o\}}>0$). Here, $o$ denotes the origin.

The {\it Wulff shape} $W(f)$ of a continuous function $f:\Sn\to (0,\infty)$ is defined to be the largest (with respect to inclusion) convex body whose support function is dominated by $f$. It is clear that $h_{W(f)}|_{\Sn}=f$, if $f$ is the restriction of a support function on $\Sn$.  Let $K,L$ be convex bodies that contain the origin. For $p\in \R$, $a,b>0$, $\lambda\in(0,1)$, define the $L_p$-Minkowski convex combination (or simply $L_p$-Minkowski sum) $(1-\lambda)\cdot K+_p\lambda \cdot L$ of $K$ and $L$ with respect to $\lambda$ by $$(1-\lambda)\cdot K+_p\lambda \cdot L:=\begin{cases}W(((1-\lambda)h_K^p+\lambda h_L^p)^{1/p}),& p\neq 0\\
W(h_K^{1-\lambda}h_L^\lambda),& p=0\end{cases}\ .$$
{One may allow $o$ to lie in the boundary of $K$ or $L$, by using the following convention: For $a,b\geq0$ and $p<0$, we set $[(1-\lambda)a^p+\lambda\beta^p]^{1/p}=0$ if $a=0$ or $b=0$.}
Note that the case $p=1$ corresponds to the classical Minkowski sum. That is, $(1-\lambda)\cdot K+_1\lambda\cdot L=(1-\lambda)K+\lambda L=\{(1-\lambda)x+\lambda y:x\in K,y\in L\}$.

    The starting point of this paper is the celebrated Brunn-Minkowski inequality (see \cite{Sch14} or \cite{Ga} or \cite{Gar02}) and its $L_p$ analogue, the so called $L_p$-Brunn-Minkowski inequality due to Firey \cite{Fir62} (see also \cite{LYZ12}), stating that for every $p\geq1$, for 
$\lambda\in(0,1)$ and for any convex bodies $K$, $L$ that contain the origin, it holds 
\begin{equation}\label{L_p-BM}V_n((1-\lambda)\cdot K+_p\lambda\cdot L)\geq \left((1-\lambda)V_n(K)^{p/n}+\lambda V_n(L)^{p/n}\right)^{n/p},\end{equation}
where $V_n(\cdot)$ denotes the volume functional in $\R^n$. 
The $L_p$-Brunn-Minkowski inequality \eqref{L_p-BM} for $p>1$ follows from the classical one (case $p=1$), a homogeneity argument and the following inclusion:
\begin{equation}\label{eq-inclusion}
(1-\lambda)\cdot K+_p\lambda\cdot L\subseteq (1-\lambda)\cdot K+_q\lambda\cdot L,\qquad -\infty<p<q\leq\infty,\qquad \end{equation}
with equality if and only if $K=L$. For $0\leq p<1$, inequality \eqref{L_p-BM} is not true in general, but it has been conjectured \cite{BLYZ12} to hold with the additional assumption of $K$ and $L$ being centrally symmetric. 
The case $p=0$ in \eqref{L_p-BM} should be understood as 
\begin{equation}\label{eq-log-BM} V_n((1-\lambda)\cdot K+_0\lambda\cdot L)\geq V_n(K)^{1-\lambda}V_n(L)^{\lambda}.\end{equation}The conjectured inequality \eqref{eq-log-BM} for symmetric convex bodies $K$ and $L$ is known as the {\it log-Brunn-Minkowski} conjecture. Due to \eqref{eq-inclusion}, \eqref{eq-log-BM} is stronger than \eqref{L_p-BM}, for $0<p<1$. 
It follows from the work of Kolesnikov, Milman \cite{KoM22} and Chen, Huang, Li, Liu \cite{CHLL20} that \eqref{L_p-BM} holds in full generality if $p\in [1-\frac{c}{n^{3/2}},1)$, where $c$ a universal constant. Other partial results include the validity of \eqref{eq-log-BM} in the plane \cite{BLYZ12} and for unconditional convex bodies \cite{Sar15} (see also \cite{Sar16}).  The infinitesimal version of \eqref{L_p-BM} (the so-called $L_p$-Minkowski inequality) has been confirmed in an isomorphic sense by Milman \cite{Milb} and Ivaki, Milman \cite{IvMi2} and in the class of zonoids independently by van-Handel \cite{vHa} and Xi \cite{Xi24}. The non-symmetric version of \eqref{L_p-BM} has been studied in \cite{XiL16} and \cite{BoK22}. Closely related results are established in \cite{CEFM04}, \cite{CER}, \cite{GaZ10}, \cite{LMNZ17},\cite{EsM21} and others. For $p<0$, \eqref{L_p-BM} is easily seen to be false even under symmetry assumptions. We refer to \cite{Bor23} for more information concerning the current status of the $L_p$-Brunn-Minkowski conjecture.

Recall the definition of {\it intrinsic volumes} of a compact convex set $K$. The $n$-th intrinsic volume of $K$ is defined to be $V_n(K)$, while for $i\in\{0,1,\dots,n-1\}$, the $i$-th intrinsic volume $V_i(K)$ of $K$ is defined by its {\it Steiner polynomial} 
$$V_n(K+\rho B^n_2)=\sum_{j=0}^n\rho^{n-j}\kappa_{n-j}V_j(K),\qquad \rho>0.$$
Here, $\kappa_d$ stands for the volume of the $d$-dimensional euclidean unit ball. We mention that the functional $V_i(\cdot)$ is positive on convex bodies and homogeneous of degree $i$. See e.g Schneider \cite[chapter 4]{Sch14} for further details. 

As with the case of volume, the classical Brunn-Minkowski inequality for intrinsic volumes (see e.g. \cite{Sch14}) together with \eqref{eq-inclusion} and the homogeneity of intrinsic volumes, yields the $L_p$-Brunn-Minkowski inequality for intrinsic volumes, for $p\geq 1$: For convex bodies $K$ and $L$ that contain the origin, it holds
\begin{equation}
    \label{L_p-BM_intrinsic}
       V_j((1-\lambda)\cdot K+_p\lambda\cdot L)\geq\left((1-\lambda)V_j(K)^{p/j}+\lambda V_j(L)^{p/j}\right)^{j/p}.
    \end{equation}
It is natural to ask if \eqref{L_p-BM_intrinsic} holds for $0\leq p<1$ for symmetric convex bodies, where as in the case $j=n$, the inequality for $p=0$ becomes
\begin{equation}\label{eq-log-BM-cms}
V_j((1-\lambda)\cdot K+_0\lambda\cdot L)\geq V_j(K)^{1-\lambda}V_j(L)^\lambda.\end{equation}
Again, due to \eqref{eq-inclusion}, the validity of \eqref{L_p-BM_intrinsic} for some $p_0\in\R$ implies the validity of \eqref{L_p-BM_intrinsic} for all $p\geq p_0$. 

The question asking for which symmetric convex bodies $K,L$ and for which $p\in[0,1)$, \eqref{L_p-BM_intrinsic} holds true was studied in \cite{BCPR}, where the authors proved the following results.
\begin{oldtheorem}\label{thmold-1}\cite{BCPR}    
Let $n\geq 3$, $p\in[0,1)$ and $j\in\{2,\dots,n-1\}$. There exists $\zeta>0$, that depends only on $n,p$, such that if $K=B_2^n$ and $h_L\in C^2
$ is even with $\|h_L-1\|_{C^2}<\zeta$ (i.e. $L$ lies in a $C^2$-neighbourhood of $B_2^n$), then \eqref{L_p-BM_intrinsic} holds.\end{oldtheorem}
\begin{oldtheorem}\label{thmold-2}\cite{BCPR}
Let $n\geq 3$, $j\in\{2,\dots,n-1\}$. There exists $0<p_j<1$, such that \eqref{L_p-BM_intrinsic} does not hold for all choises of symmetric convex bodies $K$ and $L$, for any $p\in[0,p_j)$. \end{oldtheorem}
The restriction $n\geq 3$ and $j\geq 2$ is due to the following observation (also from \cite{BCPR}): Since $V_1(K)$ is a multiple of $\int_{\Sn}h_Kd\Ha$, it follows by Minkowski's triangle inequality (H\"older's inequality if $p=0$) that the exact reverse inequality of \eqref{L_p-BM_intrinsic} holds (actually without any symmetry assumptions) in the case $j=1$, $p\in[0,1)$. That is,
\begin{equation}\label{eq-reverse-for-j=1}
V_1((1-\lambda)\cdot K+_p\lambda\cdot L)\leq \big((1-\lambda)V_1(K)^p+\lambda V_1(L)^p\big)^{1/p},    \qquad p\in(0,1),
\end{equation}
where as always, for $p=0$, the right hand side is replaced by $V_1(K)^{1-\lambda}V_1(L)^\lambda $. 

In view of Theorem \ref{thmold-2}, one might hope that (as in the case $j=n$) \eqref{L_p-BM_intrinsic} could be true if $p$ is sufficiently close to 1. We prove that, unfortunately, this is not the case. Our first main result is the following.
\begin{theorem}\label{thm-main-counterexamples}
   Let $n\geq 3$, $j\in\{2,\dots,n-1\}$, $p\in[0,1)$ and $\lambda\in(0,1)$. Then, there exist symmetric convex bodies $K$ and $L$ in $\R^n$, such that
   \begin{equation}\label{eq-thm-main-counterexamples}
    V_j((1-\lambda)\cdot K+_p\lambda \cdot L)<V_j(K)^{1-\lambda} V_j(L)^{\lambda }\leq \left((1-\lambda) V_j(K)^{p/j}+\lambda V_j(L)^{p/j}\right)^{j/p}.
    \end{equation}
\end{theorem}
On the other hand, we improve Theorem \ref{thmold-1} by removing the assumption that at least one of $K$ and $L$ is the Euclidean ball and by allowing $p$ to take negative values within some range. Below, we state our second main result. 
\begin{theorem}
   \label{positive_theorem}
     Fix $n\geq3$, $j\in\{2,\dots,n-1\}$ and $p\in(\frac{2n-j-nj}{n-1},1)$ (notice that $\frac{2n-j-nj}{n-1}<0$). There exists $\delta>0$ that depends only on $n,p$, such that for all $\lambda\in(0,1)$ and for every symmetric convex bodies $K$ and $L$, satisfying $h_K,h_L\in C^2$ and $\max\{\|h_K-1\|_{C^2},\|h_L-1\|_{C^2}\}<\delta$, inequality \eqref{L_p-BM_intrinsic} holds true. Equality holds if and only if $K$ and $L$ are dilates.
   \end{theorem}
\noindent
The proof of Theorem \ref{positive_theorem} will be a modification of the method by Kolesnikov and Milman \cite{KoM22}.

For $j\in\{1,\dots,n-1\}$ and a convex body $K$, let $S_j(K,\cdot)$ be its {\it area measure} of order $j$ (see next section for definition and basic facts concerning area measures). If $S_j(K,\cdot)$ is absolutely continuous with respect to $\Ha$ and $h_K\in C^2$, then its density $s_j(K,\cdot)$ can be computed as the normalized $j$-th elementary function of the {\it principal radii of curvature} of $K$ at a given point (see next section for more information).
Let $p\in\R$. The $L_p$ area measure of order $j$ of $K$ is defined as the measure  given by 
$$dS_{j,p}(K,\cdot)=h_K^{1-p}dS_j(K,\cdot).$$
The {\it $L_p$-Christoffel-Minkowski problem} asks the following: Given $p\in\R$, $j\in\{1,\dots,n-1\}$ and a Borel measure $\mu$ on $\Sn$, with center of mass at the origin, find necessary and sufficient conditions so that there exists a convex body $K$, such that \begin{equation}\label{eq-L_p-Christoffel-Minkowski-problem}S_{j,p}(K,\cdot)=\mu.\end{equation}And if such body exists, does it have to be unique? We remark that, usually, the case $j=n-1$ (the $L_p$-Minkowski problem) is not included in the formulation of the $L_p$-Christoffel-Minkowski problem. Our choice to include the case $j=n-1$ in \eqref{eq-L_p-Christoffel-Minkowski-problem} is for convenience in the presentation.

Under the assumption that $S_j(K,\cdot)$ is absolutely continuous with respect to $\Ha$, equation \eqref{eq-L_p-Christoffel-Minkowski-problem} can be written as
\begin{equation}\label{eq-L_pcms}
h_K^{1-p}s_j(K,\cdot)=g,    
\end{equation}where $g:\Sn\to[0,\infty)$ is a function. In fact, when we say that $K$ satisfies \eqref{eq-L_pcms}, we will always assume that $S_j(K,\cdot)$ is absolutely continuous. A function $h$ defined on $\Sn$ will be called a classical solution of \eqref{eq-L_pcms} if $h=h_K\in C^2(\Sn)$, for some convex body $K$ that satisfies \eqref{eq-L_pcms}.

The case $j=n-1$ is the {\it $L_p$-Minkowski problem} initiated (and solved) by Minkowski \cite{Min03} in the classical case $p=1$  and by Lutwak \cite {Lut93a} for general $p$. As of the existence, the $L_p$-Minkowski problem should be considered essentially solved, see e.g. \cite{BLYZ13,BBCY19, CLZ19, CY, ChW06, HLYZ2, LGWa, LGWc, Sar}. Other variants of the Minkowski problem have been extensively studied in the past years (see for instance \cite{BLYZ19, BLYZ20, FHL22, HLYZ10, HeP18, HLYZ16,HLYZ18}).
The problem of existence for $j<n-1$, remains wide open although certain important results have been obtained \cite{BIS23,Gu21,GuM03,GuX18}.

As of uniqueness, it is known that (see e.g. \cite{HLYZ2}, \cite{GuX18}) uniqueness holds if $p> 1$, $p\neq j+1$ and holds up to translation if $p=1$. Actually, it is now classical (see e.g. \cite{BLYZ12}) that 
the $L_p$-Brunn-Minkowski inequality together with its equality cases implies uniqueness for the $L_p$-Minkowski problem. Thus, if the $L_p$-Brunn-Minkowski inequality \eqref{L_p-BM} were true, then uniqueness in the symmetric $L_p$-Minkowski problem would follow.  Consequently, by \cite{KoM22} and \cite{CHLL20}, it follows that uniqueness for the symmetric $L_p$-Minkowski does hold for $p<1$ if $p$ is close to 1. In fact, the same approach works for all $j$ (see \cite[Theorem 1.6]{BCPR}). Therefore, by Theorem \ref{positive_theorem}, we have the following.
\begin{corollary}\label{cor-local-uniqueness}
There exists $\eta>0$, that depends only on $n$ and $p$, such that if $j\in\{1,\dots,n-2\}$ and $K_1,K_2$ are symmetric convex bodies such that $h_{K_i}\in C^2$ and $\|h_{K_i}-1\|_{C^2}<\eta$, $i=1,2$, then the equation $$h_{K_1}^{1-p}s_j(K_1,\cdot)=h_{K_2}^{1-p}s_j(K_2,\cdot)$$implies $K_1=K_2$.    
\end{corollary}
Results of this type have previously been obtained for the $L_p$-Minkowski problem by e.g. \cite{CLM17}, \cite{KoM22}, \cite{BCPR}. We should remark that for $j=n-1$, uniqueness does not hold in \eqref{eq-L_pcms}, even in the symmetric case \cite{JLW15}, \cite{Mila}, \cite{BoSa}.

The case where the right hand side $g$ in \eqref{eq-L_pcms} is constant has been extensively studied. For $j=n-1$, $p=0$, Firey \cite{Fir74} proved that the equation has a unique even classical solution (the constant function). Brendle, Choi and Daskalopoulos \cite{BCD17}, improving a result of Andrews \cite{And99}, proved that if $j=n-1$ and $p>-n$, then the same uniqueness result holds without the evenness assumption. Other proofs of the latter and extensions were established in \cite{Sar22} and \cite{IvMi}. Classification of solutions for $j=n-1$, $p\leq -n$ of \eqref{eq-L_pcms}, with $g\equiv 1$ and/or non-uniqueness results appear in \cite{Calabi}, \cite{Pet85}, \cite{And00}, \cite{And02}. The case $j<n-1$ was treated by Chen \cite{Che} ($p>1-j$) and McCoy \cite{McC11} ($p=1-j$) in the general (non-symmetric) case and by Ivaki, Milman \cite{IvMi} in the symmetric case but in a more general setting. We state the result below only in the symmetric case.
\begin{oldtheorem}\label{oldthm-isotropic}\cite{Che}, \cite{IvMi}
Let $j\in\{1,\dots,n-2\}$, $p\geq 1-j$ and $g\equiv1$. Then, under some regularity assumptions {on $h$ (at least $h\in C^2(\Sn)$)}, $h\equiv1$ is the unique classical solution to \eqref{eq-L_pcms}.
\end{oldtheorem}
A close inspection of the proof by Ivaki and Milman \cite{IvMi} shows that a slightly stronger statement is actually true. We state it in Theorem \ref{thm-old-D} (see next section).

Our third main result can be viewed as a stability result of Theorem \ref{oldthm-isotropic}.
\begin{theorem}\label{uniqueness_Christoffel_Minkowski}
Let $\alpha\in(0,1)$ and $p\in[0,1)$, $j\in\{1,\dots,n-2\}$, such that $(p,j)\neq (0,1)$. Then, there exists
$\delta=\delta(n,p)>0$, such that if $g$ is an even function defined on $\Sn$, then equation \eqref{eq-L_pcms} has a unique even classical solution, provided $\|g-1\|_{C^\alpha}<\delta$.
\end{theorem}

We are not aware of any results that indicate whether Theorem \ref{uniqueness_Christoffel_Minkowski} holds or not without the restriction $\|g-1\|_{C^\alpha}<\delta$.

Results in the style of Theorem \ref{uniqueness_Christoffel_Minkowski} have previously been obtained for the $L_p$-Minkowski (case $j=n-1$) problem. See Chen, Huang, Li, Liu \cite{CHLL20} for the symmetric case and Chen, Feng, Liu \cite{CFL}, B\"or\"oczky, Saroglou \cite{BoSa} for the extension to the non-symmetric case. 
The proof in the case $j<n-1$ requires some new ideas and is more elaborate than the corresponding (symmetric) case $j=n-1$. Nevertheless, the question of whether Theorem \ref{uniqueness_Christoffel_Minkowski} holds without the evenness assumption is, to the best of our knowledge, open and (in our opinion) deserves some attention. Some of the main ingredients 
for the proof of Theorem \ref{uniqueness_Christoffel_Minkowski} are Theorem \ref{thm-old-D} below by Ivaki and Milman and a non-degeneracy estimate (Section 5, Theorem \ref{lower_bound_sup_fun} below), that appears to be more general (in the convex case) than the one established in \cite{Gu21}.
\\
\\
{\bf Acknowledgement.} We are grateful to K\'aroly B\"or\"oczky, Emanuel Milman and Pengfei Guan for useful comments concerning this manuscript and Leo Brauner for providing us reference \cite{GoWeKi}. We also thank the anonymous referee for carefully reading the manuscript and for valuable comments, especially for indicating to us a shorter proof of Proposition 5.7.
\section{Preliminaries}\label{Preliminaries}
A general reference for this section is the book of Schneider \cite{Sch14}. Let $A^s=(A_{ik}^s)_{i,k=1}^N$ be $N\times N$ matrices, $s=1,\dots,N$. Their {\it mixed discriminant} $D_N(A^1,\dots,A^N)$ is defined by 
$$D_N(A^1,\dots,A^N)=\frac{1}{N!}\frac{\partial^N}{\partial t_1\dots\partial t_N}\det\left(t_1A^1+\dots+t_NA^n\right).$$
Alternatively, one can define $D_N(A^1,\dots,A^N)$ inductively. Denote by  $M^{i,k}(A^1)$ the minor that results if we exclude the $i$-th row and the $k$-th column of $A^1$. Then, 
$$D_N(A^1,\dots,A^N)=\sum_{i,k=1}^NA^1_{i,k}Q^{i,k}(A^2,\dots,A^N),$$
where $$Q^{i,k}(A^2,\dots,A^N):=\frac{(-1)^{i+k}}{N}D_{N-1}(M^{i,k}(A^2),\dots,M^{i,k}(A^N))$$
and $$D_N(A,\dots,A)=\det(A).$$Notice that $D_N(A^1,\dots,A^N)$ is positive (resp. non-negative), if $A^1,\dots, A^N$ are positive definite (resp. positive semi-definite).

Let $f:\Rn\to \R$ be a function, such that $f|_{\Rn\setminus\{o\}}\in C^2$. We denote by $D$ the usual gradient operator in $\R^n$ and by $\nabla$ the spherical gradient. That is, for $p\in\Sn$, we have $\nabla f(p)=Df(p)-\langle Df(p),p \rangle p$. Denote by $\nabla^2f$ the Hessian of $f$ on $\Sn$, with respect to the standard round metric $\delta$ on $\Sn$. We are particularly interested in the operator ${\cal A}$, given by
${\cal A}f:=\nabla^2f+\delta f$. It turns out that
if $\{e_1,\dots,e_{n-1}\} $  is a local orthonormal frame on $\Sn $, then ${\cal A}f$ can be represented with respect to $\{e_i\}$ as follows: ${\cal A}f=(\nabla_{i,k}f+\delta_{ik}f)_{i,k=1}^{n-1}$,  where $\nabla_{i,k}f$ denotes second order covariant differentiation with respect to $\{e_i\}$ and $\delta_{ik}$ is the standard Kronecker symbol. Alternatively, if $f$ happens to be positively homogeneous, for $p\in \Sn$, one has
\begin{equation}\label{eq-hessian-alter}
{\cal A}f(p)=(D_{e_i}D_{e_k}f(p))_{i,k=1}^{n-1},
\end{equation}
where $D_{e_i}$ denotes the usual directional derivative in $\R^n$.
Notice that if $f$ is an arbitrary support function (without any regularity assumption), then its second derivative exists for almost every point in $\R^n$, hence (by homogeneity) for almost every point in $\Sn$. Thus, the expression ${\cal A}f(p)$ is meaningful and positive semi-definite (due to convexity) for almost every $p\in \Sn$. 

Let $h_1,\dots,h_{n-1}:\Sn\to \R$ be support functions of class $C^2$. Define the {\it mixed curvature function} $s(h_1,\dots,h_{n-1},\cdot):\Sn\to \R$ of $h_1,\dots,h_{n-1}$ by
$$s(h_1,\dots,h_{n-1},\cdot):=D_{n-1}({\cal A}h_1(\cdot),\dots,{\cal A}h_{n-1}(\cdot)).$$
Set, also, $s_j(K_1,\cdot)=s_j(h_1,\cdot):=s(h_1[j],1[n-j-1],\cdot)$ and $s(K_1,\dots,K_{n-1},\cdot):=s(h_1,\dots,h_{n-1},\cdot)$, if $h_i=h_{K_i}$, $i,j=1,\dots,n-1$.
Denote by $\lambda_1,\dots,\lambda_{n-1}$ the eigenvalues of ${\cal A}h_{K_1}$, called the principal radii of curvature of $K_1$ at a given point $p\in\Sn$. Then, it turns out that the function $s_j$ can be computed as the normalized $j$-th elementary symmetric function of the $\lambda_i$:
$$s_j=\binom{n-1}{j}^{-1}\sum_{1\leq i_1<\dots<i_j\leq n-1}\lambda_{i_1}\dots\lambda_{i_j}.$$

For convex bodies $K_1,\dots,K_n$, their {\it mixed volume} $V(K_1,\dots,K_n)$ is defined by
$$V(K_1,\dots,K_n)=\frac{1}{n!}\sum_{k=1}^n(-1)^{n+k}\sum_{1\leq i_1<\dots<i_k\leq n}V_n(K_{i_1}+\dots+K_{i_k}).$$
Clearly, $V(K_1,\dots,K_1)=V_n(K_1)$. It is known that the mixed volume functional is positive on convex bodies, linear in each argument with respect to positive Minkowski linear combinations, monotone with respect to inclusion in each argument and invariant under permutations of its arguments. For positive integers $r_1,\dots,r_m$ that sum to $n$, we write $V(K_1[r_1],\dots, K_m[r_m])$ for the mixed volume of $r_i$ copies of $K_i$, $i=1,\dots,m$. 

It turns out that the $j$-th intrinsic volume is equal to $$V_j(K)=d_{n,j}V(K[j],B_2^n[n-j]),$$
where  $d_{n,j}=\binom{n}{j}\kappa_{n-j}^{-1}$ and $\kappa_d$ is the volume of the $d$-dimensional unit ball.

The mixed area measure $S(K_1,\dots,K_{n-1},\cdot)$ of $K_1,\dots,K_{n-1}$ is defined to be the (unique) Borel measure on $\Sn$, such that for every convex body $L$, it holds
$$V(L,K_1,\dots,K_{n-1})=\frac{1}{n}\int_{\Sn}h_LdS(K_1,\dots,K_{n-1},\cdot).$$
It is known that $S(K_1,\dots,K_1,\cdot)=S(K_1,\cdot)$. The {\textit{area measure of order}} $j$, $S_j(K_1,\cdot)$ of $K_1$ is given by $S_j(K_1,\cdot)=S(K_1[j],B_2^n[n-j-1],\cdot)$. Also, as with mixed volumes, the mixed area measure functional is linear in each argument with respect to positive Minkowski linear combinations and invariant under permutations of its arguments. In the notation of volume, intrinsic volume, mixed volume and mixed area measure we will sometimes replace convex bodies by their support functions. For instance, $V(h_{K_1},\dots h_{K_n}):=V(K_1,\dots,K_n)$.

If it happens that $h_{K_i}\in C^2(\Sn)$ and ${\cal A} h_{K_i}>0$ everywhere on $\Sn$, for $i=1,\dots,n-1$, then $S(K_1,\dots,K_{n-1},\cdot)$ is absolutely continuous with respect to $\Ha$ and its density equals $s(K_1,\dots,K_{n-1},\cdot)$. In particular, $S_j(K_1,\cdot)$ is absolutely continuous with respect to $\Ha$ with density $s_j(K_1,\cdot)$. Nevertheless (see Hug \cite{Hug99}, \cite{Hug02}), even if we do not impose any regularity assumptions on $K_1$, $s_j(K_1,\cdot)$ (defined almost everywhere) is the density function of the absolutely continuous part of $S_j(K_1,\cdot)$ in its Lebesgue decomposition with respect to $\Ha$.

For the rest of the paper, ${\cal S}^n$ will denote the class of symmetric convex bodies, such that $h_K\in C^2(\Sn)$ and ${\cal A}h_K>0$ everywhere on $\Sn$. 

The following theorem is due to Ivaki and Milman (by a small modification of the proof of \cite[Theorem 1.6]{IvMi}).
\begin{oldtheorem}\cite{IvMi}\label{thm-old-D}
Let $K\in {\cal S}^n$, such that $h_K$ is sufficiently smooth,  $j\in\{1,\dots,n-2\}$ and $p\in\R$. Set $h:=h_K$, $\sigma_j(\cdot):=\binom{n-1}{j}s_j(h,\cdot)$, $\sigma_j^{ik}:=j\binom{n-1}{j}Q^{i,k}({\cal A}h[j-1],I[n-j-1])$ and $\lambda_1,\dots,\lambda_{n-1}$ for the principal radii of curvature of $K$. Then, for arbitrary $c>0$, it holds
\begin{eqnarray*}
&&\int_{\Sn}\left(\left( c(p+1)+(j-2)h^{1-p}\sigma_j\right)h^p|\nabla h|^2+2h\sum_{i=1}^{n-1}\lambda_i\sigma_j^{ii}(\nabla_ih)^2\right)d\Ha\\
&&\leq \int_{\Sn}(h^{1-p}\sigma_j-c)h^{p+1}\left(\sigma_1-(n-1)h\right)d\Ha.\end{eqnarray*}
\end{oldtheorem}
\section{Counterexamples for any $p$ and $j$}
In this section we prove the negative Theorem $\ref{thm-main-counterexamples}$. 
The following formula concerning mixed volumes will be useful (see \cite[Theorem 5.3.1]{Sch14} for a more general statement). Let $\theta\in\Sn$, $K_1,\dots,K_{n-1}$ be convex bodies in $\R^n$ and $I$ be a segment parallel to $\theta$. Then,
\begin{equation}\label{eq-mixed-vol-proj}
V(I,K_1,\dots,K_{n-1})=\frac{V_1(I)}{n}V^{\theta^\perp}(K_1|\theta^\perp,\dots,K_{n-1}|\theta^\perp),
\end{equation} where $V^{\theta^\perp}$ denotes the mixed volume functional in the linear space $\theta^\perp$. The next lemma was proved in \cite{BLYZ12}, \cite{Sar15} for $p=0$. 
\begin{lemma}\label{lp-BM-cartesian products}
Let $p\in[0,1]$, $\lambda\in(0,1)$, $E$ be a subspace of $\R^n$, $A,C$ be symmetric convex bodies in $E$ and $B,D$ be symmetric convex bodies in $E^\perp$. Then, 
\begin{equation}\label{lp-BM-cartesian products-eqs}
(1-\lambda)\cdot(A+B)+_p\lambda \cdot(C+D)=((1-\lambda)\cdot A+_p\lambda \cdot C)+((1-\lambda)\cdot B+_p\lambda \cdot D),    
\end{equation}where $(1-\lambda)\cdot A+_p\lambda \cdot C$ and $(1-\lambda)\cdot B+_p\lambda \cdot D$ denote $L^p$-sum in $E$ and $E^\perp$ respectively.
\end{lemma}
\begin{proof}
It suffices to prove \eqref{lp-BM-cartesian products-eqs} for $p\in(0,1]$, since the case $p=0$ follows by approximation. Set 
 $$K_\lambda:=(1-\lambda)\cdot(A+B)+_p\lambda \cdot(C+D)\qquad \textnormal{and}\qquad L_\lambda:=((1-\lambda)\cdot A+_p\lambda \cdot C)+((1-\lambda)\cdot B+_p\lambda \cdot D).$$Then,
 \begin{eqnarray*}
K_\lambda&=&\Big\{x+y:(x,y)\in E\times E^\perp,\ \langle x+y,z+w\rangle\\ &&\ \ \leq \left[(1-\lambda)(h_A(z)+h_B(w))^p+\lambda (h_C(z)+h_D(w))^p\right]^{1/p},\ \forall (z,w)\in E\times E^\perp\Big\}.
 \end{eqnarray*}Notice that, since $p\in(0,1]$, for $a,b,c,d>0$, the following inequality (being a simple consequence of Minkowski's triangle inequality) holds
 $$((1-\lambda)(a+b)^p+\lambda (c+d)^p)^{1/p}\geq ((1-\lambda) a^p+\lambda c^p)^{1/p}+((1-\lambda) b^p+\lambda d^p)^{1/p}.$$
Hence,
\begin{eqnarray*}
K_\lambda&\supseteq &\Big\{x+y:(x,y)\in E\times E^\perp,\ \langle x,z\rangle+\langle y,w\rangle\\ &&\ \ \leq [(1-\lambda) h_A^p(z)+\lambda h_C^p(z)]^{1/p}+[(1-\lambda) h_B^p(w)+\lambda h_D^p(w)]^{1/p},\ \forall (z,w)\in E\times E^\perp\Big\}\end{eqnarray*}and, therefore,
 $$K_\lambda\supseteq L_\lambda.$$
To prove the reverse inclusion, let $x\in E$. Then,
$$h_{K_\lambda}(x)\leq\left[(1-\lambda)(h_A(x)+h_B(x))^p+\lambda (h_C(x)+h_D(x))^p\right]^{1/p}
=\left[(1-\lambda) h_A^p(x)+\lambda h_C^p(x)\right]^{1/p},
$$so
$$h_{K_\lambda|E}\leq h_{L_\lambda|E}=h_{(1-\lambda)\cdot A+_p \lambda \cdot C}.$$
In other words, it holds
\begin{equation}\label{eq-lp-BM-cp-inclusion1}
K_\lambda|E\subseteq     (1-\lambda)\cdot A+_p \lambda \cdot C
\end{equation}
and, similarly, one gets
\begin{equation}\label{eq-lp-BM-cp-inclusion2}
K_\lambda|E^\perp\subseteq     (1-\lambda)\cdot B+_p \lambda \cdot D.
\end{equation}
Combining \eqref{eq-lp-BM-cp-inclusion1} and \eqref{eq-lp-BM-cp-inclusion2}, we find
$$K_\lambda\subseteq (K_\lambda|E)+(K_\lambda|E^\perp)\subseteq L_\lambda.$$This completes the proof of the Lemma.\end{proof}
\begin{proof}[Proof of Theorem \ref{thm-main-counterexamples}.]
We may assume that $p\in(0,1)$.
Fix $\theta\in \Sn$. For $s>0$ and a convex body $A$ in $\theta^\perp$, set $I_s:=[-s\theta,s\theta]$ and $A'_s:=A+I_s$. Using multi-linearity of mixed volume, the fact that it is zero whenever two or more of its arguments are parallel line segments (see Schneider \cite[Theorem 5.1.8]{Sch14}) and finally \eqref{eq-mixed-vol-proj}, we obtain
\begin{eqnarray*}
V_j(A'_s)&=&d_{n,j}V((A+I_s)[j],B_2^n[n-j])\\
&=&d_{n,j}V(A[j],B_2^n[n-j])+jd_{n,j}V(I_s,A[j-1],B_2^n[n-j])\\
&=&d_{n,j}V(A[j],B_2^n[n-j])+jd_{n,j}\frac{V_1(I_s)}{n}V^{\theta^\perp}(A[j-1],(B_2^n|\theta^\perp)[n-j]).
\end{eqnarray*}
Therefore, 
\begin{equation}\label{eq-thm-main-counterexampe-general}
V_j(A'_s)=    d_{n,j}V(A[j],B_2^n[n-j])+s\frac{j2d_{n,j}}{nd_{n-1,j-1}}V_{j-1}(A)=V_j(A)+se_{n,j}V_{j-1}(A),
\end{equation}
where 
$e_{n,j}:=j2d_{n,j}/(nd_{n-1,j-1})$.

For $j\in\{1,\dots,n-1\}$, we will prove the existence of $K$ and $L$ that satisfy \eqref{eq-thm-main-counterexamples}, using induction in the dimension $n$, where $n\geq 2$. By \eqref{eq-reverse-for-j=1} (choosing non-homothetic symmetric convex bodies $K$ and $L$ with $V_1(K)=V_1(L)$), this is trivial if $n=2$. Assume that our claim holds for the positive integer $n-1$, $n\geq 3$ and for any $j\in\{1,\dots,n-2\}$. In the inductive step, fix $1\leq j\leq n-1$. It suffices to find symmetric convex bodies $K'$ and $L'$ in $\R^n$ that satisfy \eqref{eq-thm-main-counterexamples}. In fact, by \eqref{eq-reverse-for-j=1}, it suffices to assume that $j\geq 2$. By the inductive hypothesis, there exist symmetric
convex bodies $K$ and $L$ in $\theta^\perp$ that satisfy  
\begin{equation}\label{eq-V_{j-1}}
V_{j-1}((1-\lambda)\cdot K+_p\lambda \cdot L)<V_{j-1}(K)^{1-\lambda} V_{j-1}(L)^{\lambda }.
\end{equation}
Moreover, by Lemma \ref{lp-BM-cartesian products} and \eqref{eq-thm-main-counterexampe-general}, we obtain
\begin{eqnarray*}
V_j((1-\lambda)\cdot K'_s+_p\lambda \cdot L'_s)&=&V_j(((1-\lambda)\cdot K+_p\lambda \cdot L)'_s)\\
&=&se_{n,j}[V_j((1-\lambda)\cdot K+_p\lambda \cdot L)/(se_{n,j})+V_{j-1}((1-\lambda)\cdot K+_p\lambda \cdot L)].
\end{eqnarray*}
Notice that, since \eqref{eq-V_{j-1}} holds, one can take $s$ to be so large that
$$V_j((1-\lambda)\cdot K+_p\lambda \cdot L)/(se_{n,j})+V_{j-1}((1-\lambda)\cdot K+_p\lambda \cdot L)<$$
$$<\left[V_j(K)/(se_{n,j})+V_{j-1}(K)\right]^{1-\lambda}\left[V_j(L)/(se_{n,j})+V_{j-1}(L)\right]^{\lambda }.$$Hence, if $s$ is large enough, it holds
\begin{eqnarray*}
V_j((1-\lambda)\cdot K'_s+_p\lambda \cdot L'_s)&<&se_{n,j}\left[V_j(K)/(se_{n,j})+V_{j-1}(K)\right]^{1-\lambda}\left[V_j(L)/(se_{n,j})+V_{j-1}(L)\right]^{\lambda }\\
&=&[V_j(K)+se_{n,j}V_{j-1}(K)]^{1-\lambda}[V_j(L)+se_{n,j}V_{j-1}(L)]^{\lambda }\\
&=&V_j(K'_s)^{1-\lambda} V_j(L'_s)^{\lambda },
\end{eqnarray*}
where we used again \eqref{eq-thm-main-counterexampe-general}. The proof of the theorem is complete.\end{proof}
\section{$L_p$-Brunn-Minkowski inequality near the Euclidean ball}
In this section we prove {Theorem }\ref{positive_theorem}. The following simple observation will be useful for what follows.
\begin{lemma}\label{neighborhood_existence}
    Assume $M\in{\cal S}^n$ is such that $V_j(M)=1$ and ${\cal N}$ is an open $C^2$-neighborhood of $M$ in ${\cal S}^n$. Then, there exists an open $C^2$-neighborhood ${\cal N}'$ of $M$ in ${\cal S}^n$ such that, for every $K,L\in {\cal N}'$ and $\lambda\in[0,1]$, the following hold
$$i)\ (1-\lambda)\cdot K+_p\lambda\cdot L\in {\cal N},\qquad ii)\ h_{(1-\lambda)\cdot K+_p\lambda\cdot L}=\left((1-\lambda)h_K^p+\lambda h_L^p\right)^{1/p},\qquad iii)\ \frac{1}{V_j(K)^{1/j}}K\in {\cal N}.$$
\end{lemma}
\begin{proof}
The proof follows easily from the fact that if $\{K_m\},\{L_m\}$ are sequences from ${\cal S}^n$ that converge to $M$ in $C^2$ norm and $\{\lambda_m\}$ is a sequence from $[0,1]$, then $\left\{{\cal A} \left(\left((1-\lambda_m)h_{K_m}^p+\lambda_m h_{L_m}^p\right)^{1/p}\right)\right\}$ converges uniformly to ${\cal A} h_M$. 
\end{proof}

Following Kolesnikov, Milman \cite{KoM22}, but in the framework of intrinsic volumes, we define a certain second-order differential operator and study its spectrum. First, for $p\neq 0$, we will compute the second derivative of the quantity $pV(h(1+tz)^{1/p}[j],1[n-j])^{p/j}$ at $t=0$, where $h=h_K$ for some $K\in {\cal S}^n$ and $z$ is any $C^2$ function. 
We will see that, in order to prove Theorem \ref{positive_theorem}, it is enough to prove that this second derivative is non-positive, in a $C^2$-neighbourhood of $1$. 

First notice that the multilinearity of mixed volumes and a direct computation yields
\begin{eqnarray}\label{2_deriv_explicit}
    \left(\left(d_{n,j}^{-1}V_j\left(h(1+t z)^\frac{1}{p}\right)\right)^{\frac{p}{j}}\right)''\Big|_{t=0}
    &=&\frac{p-j}{p}V(h[j],1[n-j])^{\frac{p}{j}-2}V(zh,h[j-1],1[n-j])^2\nonumber\\
    &+&\frac{1-p}{p}V(h[j],1[n-j])^{\frac{p}{j}-1}V(z^2h,h[j-1],1[n-j])\nonumber\\
    &+&\frac{j-1}{p}V(h[j],1[n-j])^{\frac{p}{j}-1}V(zh[2],h[j-2],1[n-j]).
\end{eqnarray}
We define the measure $V^j_K$ (being the analogue of the cone-volume measure in the case $j=n-1$) as$$dV_K^j:=\frac{1}{n}hs(h[j-1],1[n-j])d\mathcal{H}^{n-1}.$$Define, also, the inner product
$$L^2(\Sn)\times L^2(\Sn)\ni (z_1,z_2)\mapsto \langle z_1,z_2\rangle:=\langle z_1,z_2\rangle_{L^2(V_K^j)}=\int_{\Sn}z_1z_2dV_K^j$$
and the linear operator
$$T_K^j:C^2(\Sn)\ni z\mapsto z-\frac{s(zh,h[j-2],1[n-j])}{s(h[j-1],1[n-j])}\in C^0(\Sn).$$
Observe that $L^2(V^j_{K_1})=L^2(V^j_{K_2})$, whenever $K_1,K_2\in {\cal S}^n$. Notice also that, trivially, \begin{equation}\label{eq-T_K^j-const=0}T_K^j(c)=0,\end{equation} for any constant $c\in\mathbb{R}$. We wish to rewrite \eqref{2_deriv_explicit} in terms of $V_K^j$, $T_K^j$ and $\langle\cdot ,\cdot\rangle_{L^2(V_K^j)}$. To this end, notice that  

\begin{equation}\label{linear_mixed_volume}
    V(zh,h[j-1],1[n-j])=\frac{1}{n}\int_{\Sn}zhs(h[j-1],1[n-j])d\mathcal{H}^{n-1}=\int_{\Sn}zdV_K^j
\end{equation}
and, similarly,
\begin{equation}\label{quadratic_mixed_volume}
    V(z^2h,h[j-1],1[n-j])=\int_{\Sn}z^2dV_K^j=\|z\|_{L^2(V_K^j)}^2.
\end{equation}
Finally, one can easily check that for $z_1,z_2\in C^2$, it holds
\begin{equation}\label{eq-T_K^j-symmetric}
\langle z_1,T_K^j z_2\rangle=\langle z_1,z_2\rangle-V(z_1h,z_2h,h[j-2],1[n-j]),    
\end{equation}
hence
\begin{equation}\label{bilinear_mixed_volume}
    V(zh[2],h[j-2],1[n-j])=\|z\|_{L^2(V_j)}^2-\langle z,T_K^jz\rangle .
\end{equation}
Combining \eqref{2_deriv_explicit}, \eqref{linear_mixed_volume}, \eqref{quadratic_mixed_volume} and \eqref{bilinear_mixed_volume}, we arrive at
\begin{eqnarray}\label{eq-2-deriv-1}
\left(\left(d_{n,j}^{-1}V_j\left(h(1+t z)^\frac{1}{p}\right)\right)^{\frac{p}{j}}\right)''|_{t=0}
&=&\frac{\left(d_{n,j}^{-1}V_j(h)\right)^{\frac{p-j}{j}}}{p}\left((j-p)\|z\|^2_{L^2(V_K^j)}-(j-1)\langle z,T_K^jz\rangle\right)\nonumber\\
&&+(p-j)\frac{\left(d_{n,j}^{-1}V_j(h)\right)^{\frac{p-2j}{j}}}{p}\left(\int_{\Sn}zdV_K^j\right)^2.
\end{eqnarray}
On the other hand, by \eqref{eq-2-deriv-1} and using \eqref{eq-T_K^j-const=0}, \eqref{eq-T_K^j-symmetric} and the fact that $\int_{\Sn}dV_K^j=d_{n,j}^{-1}V_j(h)$, we see that for  $c\in\mathbb{R}$, it holds $$\left(\left(d_{n,j}^{-1}V_j\left(h(1+t z)^\frac{1}{p}\right)\right)^{\frac{p}{j}}\right)''\Big|_{t=0}=\left(\left(d_{n,j}^{-1}V_j\left(h(1+t (z+c))^\frac{1}{p}\right)\right)^{\frac{p}{j}}\right)''\Big|_{t=0}.$$
Taking $z+c$ in the place of $z$ in \eqref{eq-2-deriv-1}, with $c=-(n\kappa_n)^{-1}\int_{\Sn}zdV_K^j$ and using the previous identity, we arrive at the following. 
\begin{proposition}\label{equivalent_logBM}
  Given $K\in {\cal S}^n$, $j\in\{2,\dots, n-1\}$, $p\neq 0$ and $z\in C^2$, it holds
   $$\frac{d_{n,j}^{-1}V_j(h_K)^{\frac{j-p}{j}}}{j-1}p\left(\big(V_j(h_K(1+t z)^\frac{1}{p})\big)^{\frac{p}{j}}\right)''\Big|_{t=0}=-\left(\langle z_0,T_K^jz_0\rangle_{L^2(V_K^j)}-\frac{j-p}{j-1}\|z_0\|^2_{L^2(V_K^j)}\right),$$
    where $z_0:=z-d_{n,j}V_j(K)^{-1}\int_{\Sn}zdV_K^j$ is the orthogonal projection of $z$ onto the subspace of $C^2$ functions which are orthogonal to the constants (with respect to the inner product $\langle \cdot, \cdot\rangle_{L^2(V_K^j)}$).
\end{proposition}
Before we proceed with the proof of Theorem \ref{positive_theorem}, let us discuss some properties of the operator $T_K^j$. 
Fix $z\in C^2$. 
First of all, it follows immediately by \eqref{eq-T_K^j-symmetric} that $T_K^j$ is symmetric {with respect to $\langle\cdot,\cdot\rangle_{L^2(V^j_K)}$}.
Moreover, the classical Brunn-Minkowski inequality (case $p=1$) for intrinsic volumes ensures that the function $V_j(h_K+tzh_K)^{1/j}$ is concave in $t$, for $t$ sufficiently close to 0. 
In particular, it holds $\left(V_j(h_K(1+tz))^{1/j}\right)''|_{t=0}\leq 0$. 
Thus, by Proposition \ref{equivalent_logBM}, we conclude that 
$$\langle z-c,T_K^j(z-c)\rangle \geq 0,$$where $c=d_{n,j}V_j(K)^{-1}\int_{\Sn}zdV_K^j$. However, by  \eqref{eq-T_K^j-const=0} and the symmetry of $T_K^j$, it follows that $\langle z-c,T_K^j(z-c)\rangle=\langle z,T_K^jz\rangle$. This shows that $T_K^j$ is positive semi-definite. Thus (since it is also symmetric, see e.g. Davies \cite[Theorem 4.4.5]{Dav}), the operator has a self-adjoint extension {(with respect to $\langle \cdot,\cdot\rangle_{L^2(V_K^j)}$)} {in $L^2(V_K^j)$} , which we still denote by $T_K^j$.
Notice, also, that for $z\in C^2$, we have  
\begin{equation*}
T^j_{B^n_2}z=z-\frac{s(z,1[n-2])}{s(1[n-1])}=z-D_{n-1}({\cal A}z,Id[n-2])=-\frac{1}{n-1}\Delta_{\Sn}z.
\end{equation*}
In particular, $-T^j_{B_2^n}$ is uniformly elliptic and a continuity argument shows that there exists $\delta_0=\delta_0(n,j)>0$, such $-T^j_K$ is uniformly elliptic, whenever $K\in {\cal S}^n$ satisfies $\|h_K-1\|_{C^2}<\delta_0$. Consequently, if $K$ is in a $C^2$-neighbourhood of $B_2^n$, $-T_K^j$ is elliptic and self-adjoint. Thus (see e.g. Gilbarg, Trudinger \cite[Section 8.12]{GiTr}), the spectrum of $T_K^j$ is discrete and its eigenfunctions form an orthonormal basis of $L^2$.
Moreover, the sequence of its (necessarily non-negative) eigenvalues, placed in non-decreasing order according to multiplicity, tends to infinity. 

Next, observe that $T^j_K$ maps even functions to even functions and odd functions to odd functions, while even and odd functions are mutually orthogonal with respect to  $\langle \cdot,\cdot\rangle_{L^2(V_K^j)}$. Thus, if $z$ is an eigenfunction of $T_K^j$, then its even part is also an eigenfunction, as long as $z$ is not odd. Since there is no orthonormal basis of $L^2$ consisting of odd functions and the constant 1, there exists an even eigenfunction of $T_K^j$ which is also orthogonal to 1. As in \cite{KoM22}, we denote by $\lambda_{1,e}(T_K^j)$ the smallest eigenvalue of $T_K^j$ that corresponds to even eigenfunctions which are orthogonal (with respect to $\langle \cdot,\cdot\rangle_{L^2(V_K^j)}$) to 1. Then, the well known variational characterization of eigenvalues yields
\begin{equation}\label{eq-Rayleigh quotient}
\lambda_{1,e}(T_K^j)=\min\left\{\langle z,T_K^jz\rangle_{L^2(V^j_K)}:{z\textnormal{ even},~z\perp1,~\|z\|_{L^2(V_K^j)}=1}\right\},
\end{equation}
if $\|h_K-1\|_{C^2}<\delta_0$.

\begin{lemma}\label{lemma-main}
 Let $j\in\{2,\dots,n\}$, $p\in(\frac{2n-j-nj}{n-1},1)$. Then, there exists $\delta_1>0$, that depends only on $n,p$, such that for every $K\in{\cal S}^n$ with $\|h_K-1\|_{C^2}<\delta_1$ and for every even $z\in C^2$ with $\int_{\Sn}zdV_K^j=0$, it holds
$$\langle z,T_K^jz\rangle_{L^2(V_K^j)}> \frac{j-p}{j-1}\|z\|^2_{L^2(V_K^j)}.$$
\end{lemma}
\begin{proof}
Due to \eqref{eq-Rayleigh quotient}, it suffices to prove that there exists $0<\delta_1\leq\delta_0$, such that for $K\in {\cal S}^n$ with $\|h_K-1\|_{C^2}<\delta_1$, it holds $\lambda_{1,e}(T^j_K)>(j-p)/(j-1)$. Assume that such $\delta_1$ does not exist. Then, there will be a sequence $\{K_m\}_{m\in\mathbb N}$ from ${\cal S}^n$ converging to $B^n_2$ in the $C^2$ norm, such that for every $m\in\mathbb N$, it holds 
 $$0\leq\lambda_m:=\lambda_{1,e}(T^j_{K_m})\leq\frac{j-p}{j-1}.$$
For $m\in\mathbb N$, let $z_m\in L^2$ with $\|z_m\|_{L^2(V^j_{B^n_2})}=1$ be an even eigenfunction of $T^j_{K_m}$ that corresponds to $\lambda_m$, that is $ T^j_{K_m}z_m=\lambda_mz_m$ and $\int_{\Sn}z_mdV^j_{K_m}=0$.
Since $\{z_m\}$ is a bounded sequence from $L^2$, by taking subsequences, we may assume that $\{z_m\}$ converges weakly (with respect to $L^2(V^j_{B^n_2})$) to some even $z\in L^2$ satisfying $\|z\|_{L^2(V^j_{B^n_2})}=1$ and $\{\lambda_m\}$ converges to some $\lambda\in[0,\frac{j-p}{j-1}]$.
The reader can check that $\lambda$ is an eigenvalue of $T^j_{B^n_2}=-\frac{1}{n-1}\Delta_{\Sn}$, corresponding to $z$ and the latter is even with $\int_{\Sn}zdV^j_{B^n_2}=0$.
But it is well-known (see Chavel \cite[p. 35-36]{Cha} or Taylor \cite[Corollary 4.3]{Tay}) that $\lambda_{1,e}(-\Delta_{\Sn})=2n$, so $\lambda_{1,e}(T^j_{B^n_2})=\frac{2n}{n-1}$. This implies
$$\lambda\geq\frac{2n}{n-1}=\frac{j-\frac{2n-j-nj}{n-1}}{j-1}>\frac{j-p}{j-1},$$ a contradiction. 
\end{proof}

\begin{proof}[Proof of Theorem \ref{positive_theorem}.]
First we treat the case $p\neq 0$.
Because of Proposition \ref{equivalent_logBM} and Lemma \ref{lemma-main}, 
there exists
 $\delta_1>0$ that depends only on $n,p$, such that for every $M\in{\cal S}^n$ satisfying $\|h_M-1\|_{C^2}<\delta_1$ and for every even $z \in C^2$, it holds 
\begin{equation}\label{2_deriv_negative_near_ball}
  p\left(V_j(h_M(1+t z)^\frac{1}{p})^{\frac{p}{j}}\right)''|_{t=0}\leq0,
\end{equation}with equality if and only if $z$ is constant.
According to Lemma \ref{neighborhood_existence}, there is a $\delta'>0$ (that depends only on $\delta_1$, hence only on $n,p$) such that for every $K,L\in {\cal S}^n$ satisfying $\|h_K-1\|_{C^2}<\delta'$ and $\|h_L-1\|_{C^2}<\delta'$, and for every $\lambda\in[0,1]$, we have $ (1-\lambda)\cdot K+_p\lambda\cdot L\in{\cal S}^n$, $ \|h_{(1-\lambda)\cdot K+_p\lambda\cdot L}-1\|_{C^2}<\delta_1$ and $h_{(1-\lambda)\cdot K+_p\lambda\cdot L}=((1-\lambda)h_K^p+\lambda h_L^p)^{\frac{1}{p}}$.
So, for $\lambda_0\in(0,1)$, taking $M=(1-\lambda_0)\cdot K+_p\lambda_0\cdot L$, we find
\begin{eqnarray}\label{equality_investigation}
    p\left(V_j\big((1-\lambda)\cdot K+_p\lambda\cdot L\big)^{\frac{p}{j}}\right)''\big|_{\lambda=\lambda_0}\nonumber
    &=&p\left(V_j\big(((1-\lambda)h_K^p+\lambda h_L^p)^{\frac{1}{p}}\big)^{\frac{p}{j}}\right)''\Big|_{\lambda=\lambda_0}\nonumber\\
    &=&p\left(V_j\left(\big((1-\lambda_0)h_K^p+\lambda_0h_L^p-\lambda h_K^p+\lambda h_L^p\big)^{\frac{1}{p}}\right)^{\frac{p}{j}}\right)''\Big|_{\lambda=0}\nonumber\\
    &=&p\left(V_j\left(h_M\left(1+\lambda\frac{h_L^p-h_K^p}{h_M^p}\right)^{\frac{1}{p}}\right)^{\frac{p}{j}}\right)''\Bigg|_{\lambda=0}\leq0,
\end{eqnarray}
as we can see by setting 
$z=\frac{h_L^p-h_K^p}{h_M^p}$, $\lambda=t$
and applying \eqref{2_deriv_negative_near_ball}. 
This shows that the function $\phi$ defined in $[0,1]$ as
$\phi(\lambda):=pV_j((1-\lambda)\cdot K+_p\lambda\cdot L)^{\frac{p}{j}}$ is concave, which implies \eqref{L_p-BM_intrinsic}. {Now assume that equality holds in \eqref{L_p-BM_intrinsic} for some $\lambda\in(0,1)$. Then $\phi(\lambda)=(1-\lambda)\phi(0)+\lambda\phi(1)$. Since $\phi$ is concave, this implies that $\phi$ is affine on $[0,1]$. Therefore $\phi''=0$ everywhere in $[0,1]$.
Thus, equality holds in \eqref{equality_investigation} for all $\lambda_0\in[0,1]$. But then, by the equality cases in \eqref{2_deriv_negative_near_ball}, $z:=(h_L^p-h_K^p)/h_M^p$ is constant, which proves that $K,L$ are dilates.} 

We are left with the case $p=0$. This will follow from the case $p<0$. Set {$$p_0:=\frac{(2n-j-nj)/(n-1)}{2}<0$$}(say) and let $\delta>0$ be the constant that corresponds to the case $p=p_0$ in Theorem \ref{positive_theorem}. According to Lemma  \ref{neighborhood_existence} (iii), there exists $\widetilde{\delta}>0$, such that if $K,L\in{\cal S}^n$ be such that $\|h_K-1\|_{C^2},\|h_L-1\|_{C^2}<\widetilde{\delta}$, then $\|h_{\overline{K}}-1\|_{C^2},\|h_{\overline{L}}-1\|_{C^2}<\delta$, where $\overline{K}:=K/V_j(K)^{1/j}$ and $\overline{L}:=L/V_j(L)^{1/j}$. But then we know that for $\lambda\in(0,1)$, it holds
$$V_j((1-\lambda)\cdot \overline{K}+_{p_0}\lambda\cdot\overline{L})\geq \left((1-\lambda)V_j(\overline{K})^{p_0/j}+\lambda V_j(\overline{L})^{p_0/j}\right)^{j/p_0}=1.$$
The assertion follows from \eqref{eq-inclusion} (with $p=p_0$, $q=0$ and with equality if and only if $\overline{K}=\overline{L}$) and from the fact that \eqref{eq-log-BM-cms} is invariant under independent dilations of $K$ and $L$.
\end{proof}
\section{A uniqueness result for the $L_p$-Christoffel-Minkowski problem}
The aim of this section is to prove the uniqueness result Theorem \ref{uniqueness_Christoffel_Minkowski}.
The first ingredient for its proof is an $L^\infty$ estimate for $\log h_K$, where $K$ satisfies the assumptions of Theorem \ref{uniqueness_Christoffel_Minkowski}. A similar result appears in \cite{Gu21} for $0<p<1$. We would like, however, to emphasize the fact that the constant in our estimate depends only in the $L^\infty$ norm of $\log g$, while the estimate in \cite{Gu21} depends on its $C^1$ norm.
\begin{theorem}\label{lower_bound_sup_fun}
    Fix $p\in[0,1)$ and $j\in\{1,\dots,n-2\}$, such that $(p,j)\not=(0,1)$. Assume that $K$ is a symmetric convex body satisfying \eqref{eq-L_pcms} for some {non-negative} even measurable function $g$, with $\|\log g\|_{L^\infty}<\infty$. Then there exists a constant $C>1$ that depends only on $n,p, \|\log g\|_{L^\infty}$, such that 
    \begin{equation*}
       C^{-1}<h_K(x)<C, \qquad\forall x\in \Sn.
    \end{equation*}In fact, the upper bound holds without the restriction $(p,j)\neq (0,1)$.
\end{theorem}
\begin{remark}
As will be clear from the proof, a version of Theorem \ref{uniqueness_Christoffel_Minkowski} also holds for $p=1$; in fact, due to the invariance of $s_j(K,\cdot)$ under translations, the symmetry assumption in this case can be dropped: Let $K$ be a convex body in $\R^n$, satisfying $s_j(K,\cdot)=g$ in the sense of \eqref{eq-L_pcms}, for some $j\in\{1,\dots,n-2\}$ and for some measurable function $g:\Sn\to(0,\infty)$ with $\|\log g\|_{L^\infty}<\infty$. Then, the inradius $r$ and the outradius $R$ of $K$ satisfy $C^{-1}<r\leq R<C,$
for some constant $C>1$ that depends only on $n$ and $\|\log g\|_{L^\infty}$. 
\end{remark}

The lower bound will not be needed for the proof of Theorem \ref{uniqueness_Christoffel_Minkowski}. However, in our opinion, it is of some interest, as one can prove regularity results in the spirit of \cite{Gu21}, \cite{GuM03}, \cite{GuX18} that include the case $p=0$, which, to the best of our knowledge, are new. For instance, the maximum principle (applied similarly as in \cite[Proposition 4.2]{GuX18}) along with \cite[Theorem 3.3]{GuM03} yield easily the following:
\begin{corollary}
Fix  $0<\alpha< 1$, $j\neq1$, let $g\in C^2(\Sn)$ be an even positive function and let $p=0$. Then, there exists a positive constant $C$ depending only on $n$, $\alpha$, $\min g$ and $\|g\|_{C^2}$ such that if $h$ is an even classical (convex) $C^4$ solution of \eqref{eq-L_pcms}, then $\|h\|_{C^{2,\alpha}}\leq C$.
\end{corollary}
Concerning Theorem \ref{lower_bound_sup_fun}, we first establish the upper bound. 
\begin{lemma}\label{upper_bound_sup_fun} Let $p\in[0,1)$ and $j\in\{1,\dots,n-1\}$. Under the assumptions of Theorem \ref{lower_bound_sup_fun}, there exists a constant $C>0$, that depends only on $n,p, \|\log g\|_{L^\infty}$, such that    
\begin{equation*}
h_K(x)<C, \qquad\forall x\in \Sn.    \end{equation*}
\end{lemma}
\begin{proof}
  Assume that this is not the case. Then there is a sequence $\{K_m\}$ of symmetric convex bodies with $h_m:=h_{K_m}$ and $s_{j,m}:=s_j(K_m,\cdot)$ satisfying $h_m^{1-p}s_{j,m}=g_m$ for some even $g_m>0$,  such that 
    $$\frac{1}{c}\leq g_m\leq c$$ and a sequence of unit vectors $\{x_m\}$, such that 
    $$c_m:=\max_{x\in\Sn}h_m(x)=h_m(x_m)\to\infty.$$
We have\begin{equation}\label{V_j+1_lower_bound}
\frac{n}{d_{n,j+1}}V_{j+1}(K_m)=nV(K_m[j+1],B^n_2[n-j-1])=\int_{\Sn}h_ms_{j,m}d\Ha\geq\frac{1}{c}\int_{\Sn}h_m^pd\Ha.   
\end{equation}
On the other hand, 
$$\int_{\Sn}s_{j,m}d\Ha=\int_{\Sn}g_mh_m^{p-1}d\Ha$$  implies 
\begin{equation}\label{V_j_upper_bound}
\frac{n}{d_{n,j}}V_j(K_m)\leq c\int_{\Sn}\frac{1}{h_m^{1-p}}d\Ha .
\end{equation}
Let ${\cal E}_m$ be the John ellipsoid of $K_m$. Then, for $x\in \Sn$, $h_{{\cal E}_m}(x)=|A_mx|$, for some symmetric and positive definite matrix $A_m$. Let $t_1^m\geq t_2^m\geq\dots\geq t_n^m>0$ be the eigenvalues of $A_m$, written in decreasing order. Clearly, since ${|A_mx|}\leq h_m(x)\leq \sqrt{n}{|A_mx|}$ and since $\max_{x\in\Sn}h_m(x)\to\infty$, it must hold $t_1^m\to \infty$. Also, by \cite[Lemma 3.1]{LiLiLu}, we have
\begin{equation}\label{eq-eigen}\int_{\Sn}\frac{1}{|A_mx|^{1-p}}d\Ha(x)\approx \begin{cases}
\frac{1}{t_1^m}, &0<p<1\\
\frac{1+\log t_1^m-\log t_2^m}{t_1^m}, & p=0
\end{cases},\end{equation}
where $a\approx b$ means that the quantity $a/b$ is bounded from above and from below by constants that depend only on $n$ and $p$.

\textbf{Case $0<p<1$}.\\
By symmetry, we have that for every $x\in\Sn$, $h_m(x)\geq c_m\left|\langle x_m,x\rangle\right|$, which gives 
    $$\int_{\Sn}h_m^p(x)d\Ha(x)\geq c_m^p\int_{\Sn}\left|\langle x_m,x\rangle\right|^pd\Ha(x)= c_m^p\int_{\Sn}\left|\langle u,x\rangle\right|^pd\Ha(x),$$where $u$ is a (any) fixed vector in $\Sn$. Thus,
    $$\int_{\Sn}h_m^p(x)d\Ha(x)\to\infty$$
    and at last, by \eqref{V_j+1_lower_bound},
    $$V_{j+1}(K_m)\to\infty,$$
while by \eqref{V_j_upper_bound}, \eqref{eq-eigen} and $t_1^m\to\infty$, it follows that
$$V_j(K_m)\to 0.$$
This is a contradiction because of the following special case of the Aleksandrov-Fenchel inequality (see \cite[Chapter 7]{Sch14}) 
\begin{equation}\label{Aleksandrov-Fenchel}
V_{j+1}(K_m)\leq C_{n,j}V_j(K_m)^{\frac{j+1}{j}},
\end{equation}
where $C_{n,j}=\kappa_n^{-\frac{1}{j}}\kappa_{n-j-1}^{-1}\kappa_{n-j}^{\frac{j+1}{j}}\binom{n}{j+1}\binom{n}{j}^{-\frac{j+1}{j}}$.
 
\textbf{Case $p=0$}.\\
In this case, \eqref{V_j+1_lower_bound} takes the form\begin{equation}\label{bounded_away_from_0}
       d_{n,j+1}^{-1}V_{j+1}(K_m)\geq\frac{\kappa_n}{c},
    \end{equation}
while it follows again from \eqref{V_j_upper_bound} and \eqref{eq-eigen} that 
$$V_j(K_m)\leq \frac{cd_{n,j}}{n}\int_{\Sn}\frac{1}{h_m}d\Ha\leq \frac{cd_{n,j}}{n}\int_{\Sn}\frac{1}{|A_mx|}d\Ha(x)\approx\frac{1+\log(t_1^m)-\log(t_2^m)}{t_1^m}\approx-\frac{\log(t_2^m)}{t_1^m}.$$
If $-\frac{\log(t_2^m)}{t_1^m}\to0$, then, utilizing again \eqref{Aleksandrov-Fenchel}, we immediately get 
$$V_{j+1}(K_m)\to0.$$ 
 If $-\frac{\log(t_2^m)}{t_1^m}\not\to0$, then by taking subsequences we may assume that there is $c'>0$ such that for all $m$, $$-\frac{\log(t_2^m)}{t_1^m}\geq c'$$ and hence $t_i^m\to0$ for any $i=2,\dots,n$.

Without loss of generality, we may assume that the principal axes of ${\cal E}_m$ are the coordinate axes. Consider the parallelepiped $\mathcal{P}_m:=\prod_{i=1}^n{[-t_i^m,t_i^m]}$. Since $\mathcal{P}_m\supseteq \mathcal{E}_m$, we have
\begin{eqnarray*}V_{j+1}(K_m)&\approx& V_{j+1}(\mathcal{E}_m)\leq V_{j+1}(\mathcal{P}_m)=2^{j+1}\sum_{1\leq l_1<\dots<l_{j+1}\leq n}t_{l_1}^mt_{l_2}^m\dots t_{l_{j+1}}^m\\
&\leq&2^{j+1}\binom{n}{j+1}t_1^m(t_2^m)^j\leq-\frac{2^{j+1}\binom{n}{j+1}\log(t_2^m)(t_2^m)^j}{c'}\to0.\end{eqnarray*}
Hence, in both cases we conclude that $V_{j+1}(K_m)\to0$, which contradicts \eqref{bounded_away_from_0}.
\end{proof}
Now consider a lower-dimensional compact convex set $K\subseteq E$, for some proper subspace $E$ of $\R^n$. 
We distinguish between $S_j(K,\cdot):{\cal B}(\Sn)\to [0,\infty)$ (the area measure of order $j$ of $K$, when regarded as a compact convex set in $\Rn$) and $S_{j}^{E}(K,\cdot):{\cal B}(\Sn\cap E)\to [0,\infty)$ (the area measure of order $j$ of $K$, when regarded as a compact convex set in $E$). We, also, set $s_j^E(K,\cdot)$ to be the normalized elementary symmetric function of order $j$ of the principal radii of curvature of $K$ (when these are well defined), when $K$ is regarded as a set in $E$. Then, $s_j^E(K,\cdot)$ is the density of the absolutely continuous part of $S_j^E(K,\cdot)$. Fix $N\in \Sn$. For the rest of this section $\Tilde{u}$ will stand for the geodesic (spherical) projection $\Tilde{u}=\frac{u|N^\perp}{|u|N^\perp|}$ of an arbitrary unit vector $u\neq\pm N$ onto the equator $N^{\perp}\cap\Sn$.  
\begin{proposition}\label{geodesic_proj}
Fix $j\in\{1,\dots,n-2\}$ and assume that $K\subseteq N^{\perp}$ (for some $N\in\Sn$) is a compact convex set in $N^{\perp}$. Let $u\in\Sn\setminus\{\pm N\}$ be a point, such that ${\cal A}h_K(\Tilde{u})$ is well defined. Then, ${\cal A}h_K(u)$ is also well defined and the following identity holds 
\begin{equation}\label{eq-geod-proj}s_j(K,u)=\frac{1-\frac{j}{n-1}}{\langle u,\Tilde{u}\rangle^j}s_j^{N^{\perp}}(K,\Tilde{u}).\end{equation}
\end{proposition}
A slightly weaker (but still sufficient for our purposes) version of Proposition \ref{geodesic_proj}, stating that \eqref{eq-geod-proj} holds for almost every $u\in \Sn$, could follow from equation (4.3) and Theorem 5.2 in \cite{GoWeKi}. We choose to state and prove Proposition \ref{geodesic_proj} for transparency reasons and for completeness.
\begin{proof}[Proof of Proposition \ref{geodesic_proj}.]
We write $h:=h_K$. Consider an orthonormal basis $\{e_1,\dots,e_n\}$ of $\R^n$ that satisfies $e_n=u$ and $N\in \textnormal{span}\{e_{n-1},e_n\}$. Obviously, $\{\Tilde{u},N\}$ is an orthonormal basis of $\textnormal{span}\{e_{n-1},e_n\}$ and $\{e_1,\dots,e_{n-2}\}$ is an orthonormal basis of $\Tilde{u}^\perp\cap N^\perp$. 
For a function $\varphi:\R^n\to\R$ and vectors $x,v\in\R^n$, we abbreviate $\varphi_v(x):=D_v\varphi(x)$ (the directional derivative of $\varphi$ at $x\in\Rn$ along the direction $v$) and, more shortly, $\varphi_i(x):=\varphi_{e_i}(x)$, $i=1,\dots,n$.
{In the same way, if $\varphi\in C^2(\Rn)$, we abbreviate $\varphi_{i,k}(x):=D_{e_k}\varphi_i(x)$.} Note first that for every $x\in\Rn$ and every $t\in\R$, $$h(x+tN)=\max\{\langle y,x+tN\rangle:y\in K\}=\max\{\langle y,x\rangle:y\in K\}=h(x).$$
This implies that if $h$ is twice differentiable at $x\in \Sn$, then
$$h_{n-1}(x)=\langle e_{n-1},N\rangle h_N(x)+\langle e_{n-1},\Tilde{u}\rangle h_{\Tilde{u}}(x)=\langle e_{n-1},\Tilde{u}\rangle h_{\Tilde{u}}(x),$$
hence, for $i=1,\dots,n-2$, we have 
$$h_{n-1,i}(x)=\langle e_{n-1},\Tilde{u}\rangle h_{\Tilde{u},i}(x)$$
and $$h_{n-1,n-1}(x)=\langle e_{n-1},\Tilde{u}\rangle h_{\Tilde{u},n-1}(x)=\langle e_{n-1},\Tilde{u}\rangle\big(\langle e_{n-1},N\rangle h_{\Tilde{u},N}(x)+\langle e_{n-1},\Tilde{u}\rangle h_{\Tilde{u},\Tilde{u}}(x)\big)=\langle e_{n-1},\Tilde{u}\rangle^2h_{\Tilde{u},\Tilde{u}}(x),$$
Now, by \eqref{eq-hessian-alter}, ${\cal A}h(x)$ (with respect to the orthonormal basis $e_1,\dots,e_{n-1}$ of $u^{\perp}$) can be expressed as follows
$${\cal A}h(x)=\left[\begin{matrix}
  h_{1,1} & \dots & h_{1,n-2} & \langle e_{n-1},\Tilde{u}\rangle h_{1,\Tilde{u}}\\
  .& &.&.\\.& &.&.\\.& &.&.\\
  h_{n-2,1} & \dots & h_{n-2,n-2} & \langle e_{n-1},\Tilde{u}\rangle h_{n-2,\Tilde{u}}\\
  \langle e_{n-1},\Tilde{u}\rangle h_{\Tilde{u},1}& \dots & \langle e_{n-1},\Tilde{u}\rangle h_{\Tilde{u},n-2} & \langle e_{n-1},\Tilde{u}\rangle^2 h_{\Tilde{u},\Tilde{u}}
\end{matrix}\right](x).$$
In particular, since for every $i,k$ it is true that $h_{i,k}(u)=h_{i,k}(\langle u,\Tilde{u}\rangle \Tilde{u})=\frac{1}{\langle u,\Tilde{u}\rangle}h_{i,k}(\Tilde{u})$, we conclude that ${\cal A}h(u)$ is well defined and
$${\cal A}h(u)=\frac{1}{\langle u,\Tilde{u}\rangle}\left[\begin{matrix}
  h_{1,1} & \dots & h_{1,n-2} & \langle e_{n-1},\Tilde{u}\rangle h_{1,\Tilde{u}}\\
  .& &.&.\\.& &.&.\\.& &.&.\\
  h_{n-2,1} & \dots & h_{n-2,n-2} & \langle e_{n-1},\Tilde{u}\rangle h_{n-2,\Tilde{u}}\\
  \langle e_{n-1},\Tilde{u}\rangle h_{\Tilde{u},1}& \dots & \langle e_{n-1},\Tilde{u}\rangle h_{\Tilde{u},n-2} & \langle e_{n-1},\Tilde{u}\rangle^2 h_{\Tilde{u},\Tilde{u}}
\end{matrix}\right](\Tilde{u})=\frac{1}{\langle u,\Tilde{u}\rangle}{\cal A}h(\Tilde{u}).$$
But also 1-homogeneity of $h$ yields 0-homogeneity of its first-order partial derivatives, which immediately gives $$h_{1,\Tilde{u}}(\Tilde{u})=\dots=h_{n-2,\Tilde{u}}(\Tilde{u})=h_{\Tilde{u},\Tilde{u}}(\Tilde{u})=0.$$
Let ${\cal A}^{N^\perp}$ be the operator ${\cal A}$, defined on functions from the sphere $N^\perp\cap \Sn$. Notice that with respect to the orthonormal basis $\{e_1,\dots,e_{n-2}\}$ of $N^\perp\cap \Tilde{u}^\perp$, we have ${\cal A}^{N^\perp} (h|_{N^\perp})(\Tilde{u})=(h_{ik}(\Tilde{u}))_{i,k=1}^{n-2}$, where $h|_{N^\perp}$ is the support function of $K$, when regarded as a compact convex set in $N^\perp$.

We conclude that if $\lambda_1\geq\dots\geq\lambda_{n-1}$ are the eigenvalues of ${{\cal A}}h(\Tilde{u})$, then $\lambda_{n-1}=0$ and $\lambda_1,\dots,\lambda_{n-2}$ are the eigenvalues of ${{\cal A}^{N^\perp}} (h|_{N^\perp})(\Tilde{u})$ and, also, the eigenvalues of ${\cal A}h(u)$ are precisely the numbers $\lambda_1/\langle u,\Tilde{u}\rangle,\dots, \lambda_{n-1}/\langle u,\Tilde{u}\rangle$. Thus,
\begin{eqnarray}\label{reduction_to_equator}
    s_j(K,u)
  &=&\frac{1}{\langle u,\Tilde{u}\rangle^j\binom{n-1}{j}}\sum_{1\leq i_1<\dots<i_j\leq n-1}\lambda_{i_1}\dots\lambda_{i_j}\nonumber\\
  &=&\frac{1}{\langle u,\Tilde{u}\rangle^j\binom{n-1}{j}}\sum_{1\leq i_1<\dots<i_j\leq n-2}\lambda_{i_1}\dots\lambda_{i_j}\nonumber\\
&=&\frac{\binom{n-2}{j}s^{N^\perp}_j(K,\Tilde{u})}{\langle u,\Tilde{u}\rangle^j\binom{n-1}{j}}\nonumber=\frac{1-\frac{j}{n-1}}{\langle u,\Tilde{u}\rangle^j}s^{N^\perp}_j(K,\Tilde{u}).
\end{eqnarray}
\end{proof}

\begin{proof}[Proof of Theorem \ref{lower_bound_sup_fun}]
Assume that the assertion of the theorem is not true. Then there will be a sequence $\{K_m\}$ of symmetric convex bodies with (according to Lemma \ref{upper_bound_sup_fun}) uniformly bounded support functions satisfying $h_{K_m}^{1-p}s_j(K_m,\cdot)=g_m$ with $c^{-1}<g_m<c$, for some fixed constant $c>0$ and $\min_{x\in\Sn}h_{K_m}(x)\to0$. By taking subsequences, we may assume that $\{K_m\}$ converges to some symmetric compact convex set $K$ that is contained in a subspace $N^\perp$, for some direction $N\in\Sn$ and $\{g_m\}$ converges weakly to some measurable function $g:\Sn\to\R$ that satisfies $c^{-1}\leq g\leq c$. On the other hand, we know (see \cite[Chapter 4]{Sch14}) that the sequence of measures $S_{j}(K_m,\cdot)=s_j(K_m,\cdot)d\Ha$ converges weakly to $S_j(K,\cdot)$, thus (since $h_{K_m}\to h_K$ uniformly) we conclude that $g_md\Ha=h_{K_m}^{1-p}s_j(K_m,\cdot)d\Ha\to h_K^{1-p}S_j(K,\cdot)$, weakly. From this we arrive at
$$h_K^{1-p}s_j(K,\cdot)=g,$$in the sense of \eqref{eq-L_pcms}. 

Next, notice that $s_{j}^{N^{\perp}}(K,\cdot)$ cannot be zero almost everywhere on $N^\perp\cap\Sn$, otherwise by Proposition \ref{geodesic_proj}, this would imply that $s_j(K,\cdot)$, and hence $g$, is zero almost everywhere on $\Sn$. This means that there is a constant $c'>0$ such that the set
$$\Omega:=\{x\in\Sn\cap N^{\perp}:h_K^{1-p}(x)s_j^{N^{\perp}}(K,x)>c'\}$$ has positive $\mathcal{H}^{n-2}$ measure. Consequently, for $t\in(0,1]$, the set
$$\Omega_t:=\{u\in\Sn:\Tilde{u}\in\Omega,~\langle u,\Tilde{u}\rangle<t\}$$ also has positive $\Ha$ measure.
Applying Proposition \ref{geodesic_proj} and using the fact that $h_K(u)=\langle u,\Tilde{u}\rangle h_K(\Tilde{u})$, 
we see that for every $u\in\Omega_t$, it holds
$$g(u)=\langle u,\Tilde{u}\rangle^{1-p}h_K(\Tilde{u})^{1-p}s_j(K,u)=\frac{n-j-1}{n-1}\langle u,\Tilde{u}\rangle^{1-p-j}h_K^{1-p}(\Tilde{u})s_j^{N^{\perp}}(K,\Tilde{u})>c'\frac{n-j-1}{n-1}t^{1-p-j},$$
where we used the fact that $1-p-j<0$. But then, this lower bound may become as large as we wish, thus $\textnormal{ess}\sup g=\infty$, which is a contradiction.
\end{proof}
One of the difficulties in the study of the Christoffel-Minkowski problem is the lack of satisfactory regularity results like the ones that are valid for the classical Minkowski problem, developed mainly by Caffarelli \cite{Caf90a}, \cite{Caf90b}. The following lemma (in particular, part (ii)) will serve, for our purposes, as substitute of Caffarelli's regularity theory.

\begin{lemma}\label{lemma-regularity-sub}Let $0<a<1$, $j\in\{1,\dots,n-2\}$. Then, the following hold.
\begin{enumerate}[i)]
    \item For $p\in[0,1]$, there exists a positive constant $\delta_p$, such that if $g:\Sn\to \R$ is even and $\|g-1\|_{C^\alpha}<\delta_p$, then there exists an even $C^{2,\alpha}$-classical solution to \eqref{eq-L_pcms} (with these particular $p$ and $g$).
    \item There exist positive constants $\delta_1,\ C$, that depend only on $n,a$, such that if $K$ is a symmetric convex body whose $j$-th area measure is absolutely continuous with respect to the Haar measure and $\|s_j(K,\cdot)-1\|_{C^\alpha}<\delta_1$, then $h_K\in C^{2,\alpha}$ and $\|h_K\|_{C^{2,\alpha}}\leq C$.\end{enumerate}
\end{lemma}
\begin{proof}
The assertion will be an easy consequence of the inverse function theorem for Banach spaces.
Consider the map $$\Xi_p:C_e^{2,a}\ni h\mapsto h^{1-p}s_j(h,\cdot)\in C_e^a,$$
where $C^{2,\alpha}_e$ (resp. $C^{\alpha}_e$) is the space of $C^{2,\alpha}$ (resp. $C^{\alpha}$) even functions on $\Sn$.
Then, its Fr\'echet derivative at $h=1$ exists and is given by $$T(w)=\frac{j}{n-1}\Delta_{\Sn}w+(j-p+1)w.$$Notice that $(n-1)(j-p+1)/j$ cannot be an eigenvalue of $-\Delta_{\Sn}$ that corresponds to even eigenfunction. {This implies that $T$ has trivial kernel.
At the same time, the operator $\Delta_{\Sn}$ is uniformly elliptic with smooth coefficients. 
By Schauder estimates, Fredholm alternative applies (see e.g. \cite[Theorem 6.15]{GiTr} and the discussion before it) and hence the equation $T(w)=f$ has a unique solution $w\in C^{2,\alpha}(\Sn)$ provided $f\in C^\alpha(\Sn)$. Consequently, $T$ is a bounded linear isomorphism.}

Hence, by the inverse function theorem for Banach spaces {\cite[Theorem 1.1.7.]{Hor}}, there exists a $C^{2,\alpha}$-neighbourhood ${\cal N}_p$ of 1 (depending on $p$), such that $\Xi_p|_{{\cal N}_p}:{\cal N}_p\to \Xi_p({\cal N}_p)$ is a diffeomorphism. This already proves (i). In fact, by choosing ${\cal N}_1$ to be smaller if necessary, we may assume that $h$ is a support function, for any $h\in {\cal N}_1$. Notice, also, that $\Xi_1(1)=1$, thus $\Xi_1({\cal N}_1)$ is an open neighbourhood of 1 in $C^\alpha$. We conclude that for any $g\in \Xi_1({\cal N}_1)$, there exists a symmetric convex body $K$, such that $h_K\in {\cal N}_1$ and $s_j(K,\cdot)=g$. In particular $h_K\in C^{2,\alpha}$ and $\|h_K\|_{C^{2,\alpha}}\leq C$. Part (ii), now,  follows from the fact that area measures uniquely determine convex bodies (up to translation).  
\end{proof}
We will make use of the following elementary fact. If $K$ is a convex body, such that $ K\subseteq RB_2^n$, for some $R>0$, then 
\begin{equation}\label{eq-lipshcitz}
\|h_K\|_{C^{0,1}}\leq R    
\end{equation}
\begin{lemma}\label{lemma-C^a-convergence}
Let $\{h_m\}$ be a sequence of support functions that converges in $C^0$ to a  support function $h$. Then, for $0<a<1$, $h_m\to h$ in $C^\alpha$ norm.    
\end{lemma}
\begin{proof}
By assumption, there is $R>0$, such that {$h\leq R$ and $h_m\leq R$, for all $m$. Then, by \eqref{eq-lipshcitz}, it holds $\|h_m\|_{C^{0,1}},\|h\|_{C^{0,1}}\leq R$. Thus, for $x,y\in\Sn$, we have}
\begin{eqnarray*}&&\frac{|(h_m(x)-h(x))-(h_m(y)-h(y))|}{|x-y|^a}\\
&\leq&
{\min\left\{\frac{|h_m(x)-h(x)|+|h_m(y)-h(y)|}{|x-y|^a},\left(\frac{|h_m(x)-h_m(y)|}{|x-y|}+\frac{|h(x)-h(y)|}{|x-y|}\right)|x-y|^{1-a})\right\}}\\
&\leq&\min\left\{\frac{2\|h_m-h\|_{C^0}}{|x-y|^a}, 2R|x-y|^{1-a}\right\}\\
&\leq&2\|h_m-h\|_{C^0}^{1-a}R^a.\end{eqnarray*}
This proves our claim.
\end{proof}

\begin{proposition}\label{prop-grad=0 a.e.}
Let $K$ be a convex body in $\R^n$. If $\nabla h_K=0$ almost everywhere on $\Sn$, then $K$ is a Euclidean ball, centered at the origin.    
\end{proposition}
{\begin{proof}Define the function 
$$\widetilde{h}:\R^n\setminus\{o\}\ni x\mapsto h_K(x/|x|)\in\R.$$
Let $x\in\Sn$ be such that $Dh_K(x)$ exists. Then, one can compute
$$D\widetilde{h}(x)=Dh_K(x)-\langle Dh_K(x),x\rangle x=\nabla h_K(x).$$
In other words, $D\widetilde{h}(x)=0$, for almost every $x\in\Sn$. Since $D\widetilde{h}(x)$ is $-1$-homogeneous, we conclude that $D\widetilde{h}=0$, almost everywhere in $\R^n\setminus\{o\}$. The proof follows from the well known fact that a locally Lipschitz function $f$ (defined in a connected domain), whose gradient is zero almost everywhere, has to be constant. The latter is based on the fact that if $g$ is a compactly supported smooth function, then by the assumption of $f$ being locally Lipschitz, the dominating convergence theorem gives
$$D(g\ast f)=g\ast Df.$$Here, $g\ast f$ denotes the convolution of $g$ and $f$.
\end{proof}}
The last auxiliary result before the proof of Theorem \ref{uniqueness_Christoffel_Minkowski} is a consequence of Theorem \ref{thm-old-D}.
\begin{lemma}\label{lemma-IM}
Let $(p,j)\neq (0,1)$ and $\{K_m\}$ be a sequence of symmetric convex bodies, such that $h_{K_m}\in C^2$, for all $m$. Assume that $K_m\to K$ in Hausdorff {distance}, for some symmetric convex body $K$ and that $h_{K_m}$ satisfies
$$h_{K_m}^{1-p}s_j(K_m,\cdot)=g_m,$$for some function $g_m$. If $g_m\to 1$ in $C^0$ norm, then $K=B_2^n$.
\end{lemma}
\begin{proof}
For a positive integer $m$, there exists $\widetilde{K}_m\in{\cal S}^n$, such that $h_{\widetilde{K}_m}\in C^\infty$ and $\|h_{K_m}-h_{\widetilde{K}_m}\|_{C^2}<1/m$.
Set $h:=h_K$, $s_1:=s_1(K,\cdot)$, $h_m:=h_{\widetilde{K}_m}$, $s_{1,m}:=s_1(\widetilde{K}_m,\cdot)$ and $w_m:=h_m^{1-p}s_j(\widetilde{K}_m,\cdot)$. 
Then, it is clear that $h_m\to h$ and $w_m\to 1$ in $C^0$ norm.

By Theorem \ref{thm-old-D} and since clearly $\sigma_j^{ii},\lambda_i\geq 0$, for $c=\binom{n-1}{j}$ we find\begin{equation}\label{eq-bla}\int_{\Sn}(p+1+(j-2)w_m)h_m^p|\nabla h_m|^2d\Ha\leq \int_{\Sn}(w_m-1)h_m^{p+1}\left((n-1)s_{1,m}-(n-1)h_m\right)d\Ha.\end{equation}
Recall that $s_{1,m}$ converges weakly to $s_1$. Since $w_m-1\to 0$ and $h_m\to h$ in $C^0$, we conclude that the right-hand side in \eqref{eq-bla} converges to 0. Also, $\nabla h_m(x)\to \nabla h(x)$, for all $x$ such that $\nabla h(x)$ exists (that is, for almost every $x$). Since by Theorem \ref{lower_bound_sup_fun} and  \eqref{eq-lipshcitz}, $|\nabla h_m|$ is uniformly bounded, it follows from the dominated convergence theorem that the left-hand side in \eqref{eq-bla} converges to $(p+j-1)\int_{\Sn}h^p|\nabla h|^2d\Ha$. This shows that $\nabla h=0$, almost everywhere in $\Sn$. Thus, by Proposition \ref{prop-grad=0 a.e.}, $K$ is a Euclidean ball centered at the origin. But then, $h_K^{1-p}s_j(K,\cdot)\equiv 1$, which by an easy calculation gives $K=B_2^n$. 
\end{proof}
\noindent
\begin{proof}[Proof of Theorem \ref{uniqueness_Christoffel_Minkowski}.] Existence follows from Lemma \ref{lemma-regularity-sub} (i). Assume that the uniqueness assertion of the theorem is not true. Then, by Corollary \ref{cor-local-uniqueness}, there exists a sequence $\{K_m\}\subseteq {\cal S}^n$ and a sequence of positive even continuous functions $\{g_m\}$, satisfying
$$h_{K_m}^{1-p}s_j(K_m,\cdot)=g_m$$ and $g_m\to1$ in $C^\alpha$ norm, while \begin{equation}\label{eq-vxvxv}\|h_{K_m}-1\|_{C^2}>\eta_1,\end{equation} for some fixed $\eta_1>0$ and for all $m$. By Theorem \ref{lower_bound_sup_fun} and by the Blaschke selection theorem, we may assume that there is a symmetric convex body $K$, such that $K_m\to K$. However, Lemma \ref{lemma-IM} shows that $K=B_2^n$, hence $h_{K_m}\to 1$ in $C^0$ norm. 

Using Lemma \ref{lemma-C^a-convergence}, we see that $g_mh_{K_m}^{p-1}\to 1$ in $C^\alpha$ norm. Hence, by Lemma \ref{lemma-regularity-sub}, we conclude that $\|h_{K_m}\|_{C^{2,\alpha}}\leq C$, if $m$ is large enough. Therefore, $\{h_{K_m}\}$ has a subsequence that converges in $C^2$ norm to some even support function $h$. Since $h_{K_m}\to 1$ in $C^0$ norm, we conclude that $h=1$. However, this contradicts \eqref{eq-vxvxv} and the proof of Theorem \ref{uniqueness_Christoffel_Minkowski} is complete.\end{proof}
We close this note with a counterexample that shows the necessity of the restriciton $(p,j)\neq(0,1)$ in the lower bound of Theorem \ref{lower_bound_sup_fun}.
\\
\\
{\bf Example.} {Fix $N\in\Sn$ and} set $B:=B_2^n\cap N^\perp$. Then, for $u\in\Sn\setminus\{\pm N\}$, by Proposition \ref{geodesic_proj} and $h_K(u)=\langle u,\Tilde{u}\rangle h_K(\Tilde{u})$,
we get
\begin{equation*}
    s_j(B,u)h_B(u)^j=(1-\frac{j}{n-1})s_j^{N^{\perp}}(B,\Tilde{u})h_B(\Tilde{u})^j=1-\frac{j}{n-1}.  
\end{equation*}
We see that $B$ satisfies \eqref{eq-L_pcms} with $p=1-j$ and $g=1-\frac{j}{n-1}$, as long as $S_j(B,\cdot)$ is absolutely continuous with respect to $\Ha$. Let us check that this is indeed true. Denote by $G(n,n-1)$ the Grassmannian manifold of all $n-1$-dimensional linear subspaces of $\Rn$ and by $\nu_{n-1}$ the Haar probability measure on $G(n,n-1)$ and let $\omega\subseteq \Sn$ be a Borel set. Using \cite[equation (4.78)]{Sch14}, we find
\begin{equation*}
    \int_{G(n,n-1)}S_j^E(B|E,\omega\cap E)d\nu_{n-1}(E)=\frac{(n-1)\kappa_{n-1}}{n\kappa_n}S_j(B,\omega).
\end{equation*}
The fact that $S_j(B,\cdot)$ is absolutely continuous follows from the fact that for $\nu_{n-1}$-almost every $E\in G(n,n-1)$, $B|E$ is a full-dimensional ellipsoid in $E$, thus $S^E_j(B|E,\cdot)$ is absolutely continuous for almost every $E$. $\square$

\noindent Konstantinos Patsalos, Department of Mathematics, University of Ioannina, Greece\\
k.patsalos@uoi.gr\\

\noindent Christos Saroglou, Department of Mathematics, University of Ioannina, Greece\\
csaroglou@uoi.gr \ \& \ christos.saroglou@gmail.com


\end{document}